\documentclass[letterpaper,12pt,titlepage,oneside,final]{book}

\newcommand{\href}[1]{#1} %

\usepackage{ifthen}
\usepackage{amsmath}
\usepackage{amsthm}
\usepackage{tikz}
\newtheorem{defi}{Definition}
\newtheorem{definition}{Definition}

\newtheorem{exer}{Exercise}
\newtheorem{lemma}{Lemma}
\newtheorem{theorem}{Theorem}
\newtheorem{note}{Note}
\newtheorem{prop}{Proposition}

\newboolean{PrintVersion}
\setboolean{PrintVersion}{false}
\usepackage{amsmath,amssymb,amstext} %

\usepackage[pdftex,pagebackref=false]{hyperref} %
\hypersetup{
    plainpages=false,       %
    unicode=false,          %
    pdftoolbar=true,        %
    pdfmenubar=true,        %
    pdffitwindow=false,     %
    pdfstartview={FitH},    %
    pdfnewwindow=true,      %
    colorlinks=true,        %
    linkcolor=blue,         %
    citecolor=green,        %
    filecolor=magenta,      %
    urlcolor=cyan           %
}
\ifthenelse{\boolean{PrintVersion}}{   %
\hypersetup{	%
    citecolor=black,%
    filecolor=black,%
    linkcolor=black,%
    urlcolor=black}
}{} %

\usepackage[automake,toc,abbreviations]{} %
\let\origdoublepage\cleardoublepage
\newcommand{\clearemptydoublepage}{%
  \clearpage{\pagestyle{empty}\origdoublepage}}
\let\cleardoublepage\clearemptydoublepage

\newcommand{\C}{{\mathbb C}}
\newcommand{\F}{{\mathbb F}}

\newcommand{\Q}{{\mathbb Q}}

\newcommand{\R}{{\mathbb R}}
\newcommand{\Z}{{\mathbb Z}}
\newcommand{\M}{{\mathcal M}}

\newcommand{\mO}{{\mathcal O}}

\renewcommand{\H}{{\mathcal H}}
\renewcommand{\L}{{\mathcal L}}

 \DeclareMathOperator{\im}{Im}

\DeclareMathOperator{\GL}{GL}

\DeclareMathOperator{\N}{N}

\DeclareMathOperator{\sep}{sep}

\DeclareMathOperator{\aut}{Gal(K^{\sep}/K)}

\newcommand{\A}{{\mathbb A}}
\newcommand{\p}{{\mathbb{P}^1(\mathbb{C})}}

\newcommand{\bZ}{{\mathbb Z}}

\newcommand{\bC}{{\mathbb C}}
\newcommand{\bR}{{\mathbb R}}

\usepackage[backend=bibtex]{biblatex}
\begin{document}

\pagestyle{empty}

\pagenumbering{roman}

\begin{titlepage}
        \begin{center}
        \vspace*{1.0cm}

        \Huge
        {\bf  Grothendieck's Classification of Holomorphic Bundles over the Riemann Sphere}

        \vspace*{1.0cm}

        \Large
        Andean Medjedovic \\

        \normalsize

        \vspace*{1.0cm}

\begin{center}\textbf{Abstract}\end{center}

In this paper we look at Grothendieck's work on classifying holomorphic bundles over $\p$.
The paper is divided into $4$ parts. The first and second part we build up the necessary background to talk about vector bundles, sheaves, cohomology, etc. The main result of the $3^{rd}$ chapter is the classification of holomorphic vector bundles over $\p$. In the $4^{th}$ chpater we introduce principal $G$-bundles and some of the theory behind them and finish off by proving Grothendieck's theorem in full generality. The goal is a (mostly) self-contained proof of Grothendieck's result accessible to someone who has taken differential geometry.

       \end{center}
\end{titlepage}

\pagestyle{plain}
\setcounter{page}{2}

\cleardoublepage

\renewcommand\contentsname{Table of Contents}
\tableofcontents
\cleardoublepage
\phantomsection    %

\pagenumbering{arabic}

\chapter{Complex Manifolds and Vector Bundles}
\section{Complex Manifolds}
\begin{definition}[Complex Manifold]
We say a manifold $M$ is a complex manifold if each of the charts, $\phi_\alpha$, map from an open subset $U_\alpha$ to an open subset of $V_\alpha \subset \bC^n$ and the transition maps
$\phi_{\alpha \beta} = \phi_{\beta}\circ\phi^{-1}_{\alpha}$ are biholomorphisms (bijective holomorphisms with a holomorphic inverse) as maps from $\phi_{\alpha}(U_{\alpha} \cap U_\beta)$ to
$\phi_{\beta}(U_{\alpha} \cap U_\beta)$.

\end{definition}

We say that the (complex) dimension of the manifold over $\bC$ is $n$. A Riemannian surface is a manifold in the special case that the dimension is $1$.

\begin{definition}[Projective Spaces]
We define the $n$ dimensional projective space over $\mathbb{C}$, $\mathbb{P}(\bC)^n$, as the set of equivalence classes of non-zero vectors in $v \in \bC^{n+1}$ under the equivalence $v \sim \lambda v$ for $\lambda \in \bC$.
\end{definition}

\begin{prop}
The $n$ dimensional projective space is indeed an $n$-dimensional complex manifold.
\end{prop}
\begin{proof}
Let $[z_1, \cdots, z_{n+1}]$ ($z_i \in \bC$, not all zero) be the equivalence class corresponding to  $(z_1, \cdots, z_{n+1}) \in \bC^{n+1}$. Let $U_i$ be the set of equivalence classes with $z_i$ non-zero. Then $U_i$ cover $\mathbb{P}(\bC)^n$. Let $\phi_i : U_i \to \bC^n$ by $[z_1, \cdots, z_{n+1}] \mapsto (\frac{z_1}{z_i},\cdots,\frac{z_{i-1}}{z_i},\frac{z_{i+1}}{z_i},\cdots, \frac{z_{n+1}}{z_i})$.

One immediately sees that $\phi_i$ is well defined and if $z_i,z_j \ne 0$ then $\phi_j \circ \phi_i^{-1}:\phi_{i}(U_i \cap U_j) \to \phi_{j}(U_{i} \cap U_{j})$ is a biholomorphism.
\end{proof}

\begin{theorem}
A holomorphic function, $f$, on a compact Riemann surface is constant.
\end{theorem}
\begin{proof}
We have that $|f(p)|$ is maximal for some $p$. Take a chart around $p$ to a neighbourhood of $0$. Then the composition of $f$ with the chart is maximal at $0$, contradicting the maximum modulus principle.
\end{proof}The Riemann sphere is defined to be the Riemann surface $\p$.

\section{Vector Bundles}

Let $M$ be a manifold, we say a (real) vector bundle $V$ over $M$ is pair of a manifold and projection map $(V, \pi)$ with $\pi: V \to M$ so that for every $p \in M$, $\pi^{-1}(p)$ is a $\bR$-vector space we have that there is some open $U$ around $p$, and a homeomorphism $\varphi_U$ with $\varphi_{U_i}: \pi^{-1}(U) \to U \times \bR^k$. We say that the rank of the vector bundle $V$ (over $\bR$) is $k$.

We can extend this definition to holomorphic vector bundles over a complex manifold in the following way:
\begin{definition}
Let $M$ be a complex manifold, we say a pair $(E,\pi)$ over $M$ with rank $k$ is a holomorphic vector bundle if for every $p \in M$, $\pi^{-1}(p)$ is a $\bC$-vector space and there is an open subset $U$ of $M$ with a biholomorphism $\varphi_U: \pi^{-1}(U) \to U \times \bC^k$. Equivalently, we can require the transition maps to $\bC$ be linear isomorphisms:

$$proj(\varphi_U \circ \varphi_V^{-1})|_{(U \cap V)\times \bC^k} : \bC^l \to \bC^l$$

\end{definition}

 Furthermore, we say a vector bundle is a line bundle if it has rank $1$ and we say a (complex) vector bundle $E$ is trivial if it is isomorphic to $\bC^k \times M$. Note that, locally, every bundle is trivial.

 \begin{definition}
 A section of a vector bundle $V$ is a continuous map $\sigma: M \to V$ so that $\pi\circ\sigma = 1_M$. The vector space of all sections on $V$ over $M$ is denoted by $\Gamma(M,V)$. If the vector bundle $E$ is holomorphic and the map $\sigma$ is holomorphic we say it is a holomorphic section and denote the corresponding vector space $H^{0}(M,\mathcal{O}(E))$.
 \end{definition}
  The reason we use this notation will become clear later on.

 \begin{prop}
 Let $E_1$ and $E_2$ be $2$ holomorphic vector bundles over a complex manifold $M$ of rank $k,l$. Then we can define the vector bundles $E_1^*$, $\det(E_1)$, $E_1 \oplus E_2$, and $E_1 \otimes E_2$.
 \end{prop}
 \begin{proof}
 Let $U,V$ be sufficiently small open sets around $p$. Let $\phi_1$, $\phi_2$ be $2$ corresponding charts for $E_1$ and similarly $\varphi_1$, $\varphi_2$ for $E_2$. Let $T_{12}$ and $\mathcal{T}_{12}$ be the linear transition maps for $E_1$ and $E_2$. Then we define the transition charts for $E_1^*$ ,
 $\det(E_1)$, $E_1 \oplus E_2$, and $E_1 \otimes E_2$:

 \begin{itemize}
     \item $T_{12}^*$
     \item $\det(T_{12})$
     \item $T_{12}\oplus \mathcal{T}_{12}$
     \item $T_{12} \otimes \mathcal{T}_{12}$
 \end{itemize}

 These are then invertible and linear and the vector bundles have rank $k, 1, k+l$, and  $kl$ respectively.
 \end{proof}

 \begin{definition}
 Let $E_1$ and $E_2$ be $2$ holomorphic vector bundles over a complex manifold $M$. Suppose we have an invertible map $f$ so that the following diagram commutes and the restriction, $f|_{\pi^{-1}(p)}: \pi^{-1}(p) \to \pi^{-2}(p)$ is linear.

\tikzset{every picture/.style={line width=0.75pt}} %
\begin{center}

\begin{tikzpicture}[x=0.75pt,y=0.75pt,yscale=-1,xscale=1]
\draw    (277.5,124) -- (349.5,123) ;
\draw [shift={(351.5,123)}, rotate = 539.23] [color={rgb, 255:red, 0; green, 0; blue, 0 }  ][line width=0.75]    (10.93,-3.29) .. controls (6.95,-1.4) and (3.31,-0.3) .. (0,0) .. controls (3.31,0.3) and (6.95,1.4) .. (10.93,3.29)   ;

\draw    (273.5,134) -- (303.35,176.37) ;
\draw [shift={(304.5,178)}, rotate = 234.82999999999998] [color={rgb, 255:red, 0; green, 0; blue, 0 }  ][line width=0.75]    (10.93,-3.29) .. controls (6.95,-1.4) and (3.31,-0.3) .. (0,0) .. controls (3.31,0.3) and (6.95,1.4) .. (10.93,3.29)   ;

\draw    (358.5,130) -- (326.65,175.36) ;
\draw [shift={(325.5,177)}, rotate = 305.07] [color={rgb, 255:red, 0; green, 0; blue, 0 }  ][line width=0.75]    (10.93,-3.29) .. controls (6.95,-1.4) and (3.31,-0.3) .. (0,0) .. controls (3.31,0.3) and (6.95,1.4) .. (10.93,3.29)   ;

\draw (314,107) node   {$f$};
\draw (274,159) node   {$\pi _{1}$};
\draw (352,159) node   {$\pi _{2}$};
\draw (266,124) node   {$E_{1}$};
\draw (364,124) node   {$E_{2}$};
\draw (314,187) node   {$M$};

\end{tikzpicture}
\end{center}
 We say that $E_1$ and $E_2$ are isomorphic. Similarly, if there is an injection from $E_1$ to $E_2$ then we say $E_1$ is a subbundle of $E_2$.

 \end{definition}

\begin{theorem}
Let $M$ be a complex manifold and consider a short exact sequence of vector bundles over $M$:
$$0 \to E_1 \xrightarrow{p} E \xrightarrow{q} E_2 \to 0$$
This sequence splits, that is, $E \cong E_1 \oplus E_2$.
\end{theorem}
\begin{proof}
We first construct an inner product over $E$.

Let $U_\alpha$ cover $M$ with $E$ trivial over each $U_\alpha$. Let $\rho_\alpha$ be a corresponding partition of unity. We can choose an inner product $\langle, \rangle _\alpha$ on each $E|_{U_\alpha}$. Extend each inner product to be $0$ outside $E|_{U_\alpha}$. Now consider the inner product given by:
$$\langle, \rangle = \sum_{\alpha}\rho_\alpha\langle, \rangle _\alpha$$
This is defined on all of $E$.

Under this inner product we can write $E = (p(E_1))^{\perp} \oplus (p(E_1))$. Note that $p(E_1) \cong E_1$ by injectivity. We also have the restriction of $q|_{(\ker q)^{\perp}}$ from $(\ker q)^{\perp} \to E_2$ is surjective (by exactness) and injective as $q(x) = q(y)$ means $x-y \in \ker(q)$. But $E_2 \cong (\ker q)^{\perp} =  (p(E_1))^{\perp}$ and so the sequence splits.
\end{proof}

\begin{lemma}
Let $L$ be a line bundle on a complex manifold $M$. Then $L$ is trivial if and only if there is some nowhere $0$ section on $L$.
\end{lemma}
\begin{proof}
Suppose $L$ is the trivial bundle. Let $x \in M$. Then the section sending $x \mapsto (x,1)$ is nowhere $0$.

Suppose we have a nowhere section $\sigma$, sending $x \to (x,\sigma(x))$. Then consider the isomorphism $f: M \times \bC \to L$ by $(x,c) \mapsto (x,c\sigma(x))$.
\end{proof}
\begin{lemma}
Let $S$ be a Riemann surface and let $E$ be a vector bundle. Then $E \cong E' \oplus I_{{\rm rank} E - 1}$ for some line bundle $E'$.
\end{lemma}
\begin{proof}
We note that if the rank of $E$ is at least $2$ then there is a section $\sigma$ that is non-zero everywhere by perturbing a section (not identically $0$) locally around its zeroes.

Take the line bundle parametrized by $\sigma$, $L \cong I_1$. We can then split $E$ as:
$$0 \to I_1 \to E \to E_1$$
Is short exact for some $E_1$.
We then do the same for $E_1$, inductively, so $E \cong E' \oplus I_{{\rm rank} E - 1}$.
\end{proof}

Once we have the notion of a degree of a line bundle we will be able to show $2$ line bundles are isomorphic if and only if their degree is the same. This, combined with the above, gives us that $E_1 \cong E_2$ if and only if they have the same rank and same degree.

\chapter{Sheaves And Cohomology}
\section{Sheaves}

\begin{definition}
Let $X$ be a topological space. For every open $U \subset X$ we associate an abelian group $\mathcal{F}(U)$ so that:
\begin{itemize}
    \item $F(\emptyset) = 0$
    \item If $V \subseteq U$ there is a group morphism $\rho_{U,V}:\mathcal{F}(U) \to \mathcal{F}(V)$
    \item $\rho_{U,U} = 1$
    \item If $W \subseteq V \subseteq U$ then $\rho_{U,W} = \rho_{V,W} \circ \rho_{U,V}$
\end{itemize}
We say that $\mathcal{F}$ is a presheaf. We can write $\rho_{U,V}(f)$ as $f|_V^{U}$
\end{definition}

\begin{definition}
Let $\mathcal{F}$ be a presheaf on $X$. Let $U$ be open with open cover $\{U_i\}$.  $\mathcal{F}$ is a sheaf if we have:
\begin{itemize}
    \item If $s\in \mathcal{F(U)}$ with $\rho_{U,U_i}(s) = 0$ for all $i$, then $s= 0$
    \item If $s_i \in \mathcal{F}(U_i)$ with (for any $i,j$):
            $$\rho_{U, U_i\cap U_j}(s_i) = \rho_{U, U_i\cap U_j}(s_j)$$
            then there is $s \in \mathcal{F}(U)$ so that $\rho_{U,U_i}(s) = s_i$
\end{itemize}

\end{definition}
\begin{prop}
Let $k \in \mathbb{Z}$. Let $E$ be a vector bundle over a complex manifold $M$. Let $X$ be a Riemann surface with $x \in X$ and $L$ be a line bundle over $X$. Then the following are sheaves:
\begin{itemize}
    \item $\mathcal{O}(E)$ where to each $U \subset M$ we associate the abelian group (under pointwise multiplication)  $H^{0}(E,U)$ with the maps $\rho_{U,V}$ being restrictions.
    \item $\mathcal{O}_M$, where to each $U \subset M$ we associate the abelian group of holomorphic functions $f: U \to \bC$
    \item $\mathcal{O}^*_M$, where to each $U \subset M$ we associate the abelian group of holomorphic functions $f: U \to \bC^*$
    \item $\mathcal{O}_X(-kx)$, where to each $U \subset X$ we associate the abelian group of holomorphic functions $f: U \to \bC$ that vanish at $x$ with multiplicity $k$,
    \item $\mathcal{O}_X(kx)$ where to each $U \subset X$ we associate the abelian group of holomorphic functions $f: U \to \bC$ that that have a pole of order $k$ at $x$.
    \item $L(-x)$, where $U$ is associated to the holomorphic sections of $L|_{U}$, vanishing at $x$.
    \item $\bC_x$, the skyscraper sheave, with $U$ associated to $\bC$ if $x\in U$ and $0$, otherwise.
 \end{itemize}
\end{prop}

\section{Cech Cohomology}
Let  $\mathfrak{U} =\{U_\alpha\}$ cover a complex manifold $X$ (for $\alpha$ in some index set $I$) and let $\mathcal{F}$ be a sheaf on $X$.

\begin{definition}
Let $$C^i = \prod_{\alpha_1, \cdots, \alpha_i \in I} \mathcal{F}(\cap_{k = 1}^{i}U_{\alpha_k})$$
Let $d_i : C^i \to C^{i+1}$ via (taking the product of the maps over the indices):
$$f_{\{\alpha_1, \cdots, \alpha_{i}\}} \mapsto \sum_{k}^{i+1}(-1)^k f|_{\cap{j}U_{\alpha_j}}^{\cap{j \ne k}U_{\alpha_j}} $$
\end{definition}
This gives rise to the Čech complex:
$$C^0 \xrightarrow{d_0}C^1 \xrightarrow{d_1} \cdots $$

\begin{exer}
One can see this forms a complex by verifying that $d_{i+1} \circ d_i = 0$
\end{exer}

\begin{definition}
We define the $p^{th}$ Čech cohomology group by taking the quotient group:
$$H^p(X, \mathfrak{U}, \mathfrak{F}) = \frac{\ker(d_p)}{\im (d_{p+1})}$$
\end{definition}

One may wonder to what extent does the cohomology depend on the open cover. It turns out that by a result due to Leray, beyond the scope of this paper, we can choose sufficiently refined coverings so the cohomology doesn't change.

\begin{note}
Our use of the notation $H^{0}(X, \mathcal{O}(E))$ before is justified as $\ker d_0 = H^{0}(X, \mathcal{O}(E)) $
\end{note}

\begin{lemma}[Induced Long Exact Sequences]
Suppose we have a short exact sequence of sheaves:
$$0 \to \mathcal{E} \to \mathcal{F} \to \mathcal{G} \to 0$$
That is, for any open set $U$, the functors at $U$ to the category of abelian groups form a short exact sequence.

Then there is an induced long exact sequence of cohomology groups:
$$0 \to H^0(X, \mathcal{E}) \to H^0(X, \mathcal{F}) \to H^0(X, \mathcal{F}) \to H^1(X,\mathcal{E}) \cdots$$
\end{lemma}
\begin{proof}
The proof is an easy but tedious application of Snake lemma twice. We will not, however, prove it here.
\end{proof}

\chapter{Line Bundles over $\p$ and Grothendieck's Theorem for Vector Bundles}
\section{The degree map}
For the rest of the paper we can fix $p \in X$.
We now turn our attention over to line bundles over $\p$. We proved earlier in the paper that any vector bundle $E$ over $X$ can be written as $L \oplus I_m$ where $L$ is a line bundle. When are $2$ line bundles isomorphic?

We leave the following proposition as an exercise.

\begin{prop}
Let $E,F,G,H$ be vector bundles. If
$$0\to E \to F\to G\to 0$$
is short exact then:
$$0\to E\otimes H \to F\otimes H\to G\otimes H\to 0$$
Is short exact as well. Furthermore, if $H^1(X,G^* \otimes E) = 0$, then $F \cong E\otimes G$.
\end{prop}

\begin{definition}
Consider the short exact sequence:
$$0 \to \mathcal{O}_X(\bZ) \to \mathcal{O}_X \xrightarrow{e^{2\pi i f}} \mathcal{O}_X(\bC^*) \to 0$$
And the induced long exact sequence:
$$\cdots H^1(X, \mathcal{O}_X) \to H^1(X, \mathcal{O}_X^*) \xrightarrow{deg} H^2(X, \mathcal{O}_X(\bZ)) \to H^2(X, \mathcal{O}_X) \cdots$$
The degree map is then defined to be $deg$.
\end{definition}
\begin{lemma}
The degree map is a bijection.
\end{lemma}
\begin{proof}
By Leray  $H^1(X, \mathcal{O}_X) \cong 0$ and by a theorem of Grothendieck's $H^2(X, \mathcal{O}_X) \cong 0$. By exactness, the degree map is a bijection.

Furthermore, by Poincaré duality, $H^2(X, \mathcal{O}_X(\bZ)) \cong H_0(X, \mathcal{O}_X(\bZ))\cong \bZ$.
\end{proof}

\section{The Classification of Vector Bundles}

\begin{prop}
$H^1(X,\mathcal{O}_X^*)$ is the set of line bundles on $X$ (up to isomorphism).
\end{prop}
\begin{proof}
Let $L$ be a line bundle.
Choose a cover fine enough so that $L$ is trivial on each intersection.
Let $\phi_j^{-1} \circ \phi_i$ be the transitions, then $\ker d_1$ is precisely the set of $\phi_j^{-1} \circ \phi_i$, as $(\phi_j^{-1} \circ \phi_i)^{-1} = \phi_i^{-1} \circ \phi_j$. Let $\varphi_i$ be another trivialization of $L$. Then $\varphi_i^{-1} \circ \varphi_j^ = f^{-1}(\phi_i^{-1} \circ \phi_j)g$ where $f,g$ are biholomorphic maps on $U_i \cap U_j$. Note that $\im(d_0)$ is the set of maps that can be written as $fg^{-1}$ for some $f:U_i \to \bC^*$, $g:U_j \to \bC^*$. So up to $\im(d_0)$, line bundles are unique elements of $\ker d_1$. The conclusion follows.
\end{proof}

\begin{prop}
$H^1(X, \bC_p) = 0$.
\end{prop}
\begin{proof}
We want to show $\ker(d_1) = 0$ so it suffices to check that $d_0(C_0) = 0$. Take any refinement with only one open set, $U_1$ containing $p$. Let $c \in\mathcal{F}(U_1)$, then $d_0(c) = 0$.
\end{proof}

We also need the following $2$ lemmas:
\begin{lemma}
$\dim H^0(X, \mathcal{O}(m)) = m+1$ if $m \geq 0$.
\end{lemma}
\begin{proof}
There are some charts $\phi_0$ on $U_1$around $0$ and some chart $\phi_1$ on $U_2$ so that $U_1 \cap U_2 \ne \empty$  of that set so that the
transition function is $\frac{1}{z^m}$.

It follows that the image of any section under the charts must have Laurent expansion:
$$\frac{1}{z^m}\sum_{k = 0}^{m}\alpha_i z^i$$
\end{proof}

Note that if $m < 0$, we only have the $0$ section.

\begin{lemma}
For any vector bundle $E$ of rank $k$ over $X$ we can find some $O(n)$ so that $E \otimes O(n)$ has a holomorphic section that is not everywhere $0$.
\end{lemma}
\begin{proof}
Let $n > \dim H^1(X, E)$.
We have a section $\sigma$ that only vanishes at $p$. This yields the following short exact sequence:

$$0 \to\ E \to E\otimes\mO(n) \xrightarrow{\sigma(p)^n} \bC_p^{kn} \to 0$$
This induces a long exact sequence with the sum of alternating dimensions being $0$, so we now have

\begin{align*}
\dim H^0(X, E \otimes O(n)) &=\dim H^1(X, E \otimes O(n)) + \dim H^0(\bC_p^{nk}) \\
                            & + \dim H^0(X,E) - \dim H^1(X,E)\\
			    &\geq nk - \dim H^1(X,E).
\end{align*}
And so $\dim H^0(X, E \otimes \mO(n)) \geq 1$.

\end{proof}
Note that this implies we can let $n$ be so that $\dim H^0(X, E \otimes \mO(n-1)) = 0$ but $\dim H^0(X, E \otimes O(n)) > 0$ (as $\dim H^0(X, E \otimes \mO(n-1)) < \dim H^0(X, E \otimes \mO(n))$).

We are finally ready to prove Grothendieck's classification of vector bundles.

\begin{theorem}[Grothendieck]
Let $E$ be a rank $k$ vector bundle over $X$. Then: $$E \cong \bigoplus_{i = 1}^{k}\mO(d_i)$$
\end{theorem}

\begin{proof}
Let $\mO(n)$ be as above for $p \in X$, arbitrary.
We can then take a holomorphic section $\sigma$ that never vanishes (If it did vanish at $p$ then $\sigma\sigma_p^{-1} \in E\otimes \mO(m-1)$ wouldn't) and thus find a trivial subbundle, $L$ of $E\otimes \mO(n)$. We let $Q$ be the quotient bundle of $L$ and $E$ and suppose by induction that it decomposes as $Q = \bigoplus_{i=1}^{k-1}\mO(b_i)$.\\

Note by Riemamm-Roch, $\dim H^1(X, \mO(-1)) = 0$.

So we have the following $2$ exact sequences (after tensoring with $\mO(-1)$):
$$0\to \mO(-1) \to \mO(E\otimes O(n-1)) \to \mO(Q(-1)) \to 0$$
$$0 \to H^0(X,\mO(Q(-1))) \to 0$$
So $b_i \leq 0$.

Note by Riemamm-Roch, $\dim H^1(X, \mO(-b_i)) = 0$.
We now calculate:
$$H^1(X, Q^*) = H^1(X, \oplus_{i = 1}^{k-1}\mO(-b_i)) = 0$$

Now consider again
$$0 \to L \to E\otimes \mO(m) \xrightarrow{\alpha} Q \to 0$$
Tensoring by $Q^*$:
$$0 \to \mO(Q^*) \to \mO({\rm Hom}(Q, E\otimes O(m))) \to \mO({\rm Hom} (Q,Q))\to 0$$
The induced cohomology has the following surjection:
$$H^0(X, {\rm Hom}(Q,E\otimes \mO(m))) \to H^0(X, {\rm Hom}(Q,Q)) \to 0$$
Thus there is some $\beta: Q \to E\otimes \mO(m)$ so that $\alpha \circ \beta = id_Q$, and by splitting lemma
$$E\otimes \mO(m) \cong L \oplus Q.$$
Tensoring
$$E \cong \mO(-m)\oplus\bigoplus_{i=1}^{k-1}\mO(-m+b_i),$$
as required.

\end{proof}

\chapter{Principal Bundles}
\section{Preliminaries}
Let $X$ be a Riemann Surface.
\begin{defi}
A fiber bundle over $X$ is a triple $(E,F,\pi)$ with $E$ and $F$ being topologies so that:
\begin{itemize}
    \item $\pi: E \to X$ is a surjection.
    \item For every $x\in X$ there is an open set $x \in U$ and a chart $\phi: \pi^{-1}(U) \to U\times F$
    so that ${\rm proj}_{U}\circ \phi(q) = \pi(q)$ for $q \in \pi^{-1}(U)$
\end{itemize}

\end{defi}
We say that $F$ is the fiber, $E$ is the total space and $\pi$ is the projection.

\begin{defi}
Let $G$ be a group. A principal $G$-Bundle, $P$,  is a fiber bundle with $G$ as its fiber. We also require a continuous right $G$-action on $P$ that is free and transitive.
\end{defi}

We mainly concern ourselves with $G$ being a Lie-group.
\begin{defi}
Let $(P,\pi)$ be a principal $G$-bundle over $X$. Let $\rho$ be a continuous action on the space of homeomorphisms of a topology $F$. Let $\rho$ be the right action given by the $(p,f)g = (pg, \rho(g^{-1}f)$. We say the associated bundle is $(P \times_{\rho} F, \pi_\rho)$ where:
\begin{itemize}
    \item P $\times_{\rho} F = P \times F /\sim $ where the equivalence classes are given by $[pg,f] = [p,\rho(g)f]$
    \item $\pi_{\rho}[p,f] = \pi(p)$
\end{itemize}
\end{defi}

\begin{defi}
Let $H$ be a subgroup of $G$. We say $P$ has a reduction to $H$, if there is a non-zero section in  $P \times_G G/H$.
\end{defi}
\begin{defi}
Let $\rho$ be a representation of $G$ into $GL(V)$. We define $P \times_G \rho: P\times_G G \to P \times_G GL(V)$.
\end{defi}

\begin{defi}
Let $G$ be a connected compact Lie group. Let $T$ be a maximal torus and $N$, its normalizer. Then we define the Weyl group to be $N/T$.
\end{defi}
\begin{defi}
We say a Lie group is reductive if its Lie algebra is reductive. We say a Lie algebra is reductive if it can be written as a direct sum of a semi-simple algebra and its center.
\end{defi}
From this point on we let $G$ be a compact Lie group, let $G_0$ be the connected component at the identity and let $\mathfrak{g}$ be its Lie-algebra. We let $H$,$N$,$W$ and $\mathfrak{h}$ be a Cartan subgroup, normalizer of a Cartan subgroup, Weyl group and the lie subalgebra of the Cartan subgroup. Let let $ad$ be the adjoint representation of $G$.

Let $P$ be a holomorphic principal $G$-bundle and $E = P \times_G ad$.

Finally, let $H^1(X,\mO_X(G))$ be the set of holomorphic $G$-bundles over $X$.

\begin{theorem}[Grothendieck's Theorem for the Orthogonal Case]
A vector bundle $V$ has an orthogonal form if and only if it is isomorphic to its dual.
\end{theorem}
 We do not prove this theorem within this paper.

\begin{lemma}
Suppose we have a holomorphic section $s$ in $E$ and there is a fiber $E_a$ so that $s(a)$ is a regular element of the lie algebra of $E_a$. Then for any $x$, $s(x)$ is regular in $E_x$.
\end{lemma}
\begin{proof}
The coefficients of the polynomial defining $ad s(x)$ must be constant as they are holomorphic functions, by compactness of $X$. Thus $s(x)$ is a regular element everywhere.
\end{proof}
\begin{lemma}
Suppose we have a section $s$ in $E$. Then we have a section in $P \times_G G/N$.
\end{lemma}
\begin{proof}
By the maximal torus theorem, any $2$ Cartan subgroups are conjugate. The kernel of the action on any particular maximal torus $T$ is $N(T)$. It follows that $G/N$ is the set of Cartan subalgebras. The section given by sending $s(x)$ to its corresponding subalgebra gives a section in $P \times_G G/N$.
\end{proof}
\begin{lemma}
Suppose we have a section in $P \times_G G/N$. Then we have a section in $P \times_G G/T$.
\end{lemma}
\begin{proof}
We first prove the Weyl group is discrete. For any torus $T$ of rank $n$ we have the following short exact sequence:
$$0 \to \bZ^n \to \bR^n \to T \to 0$$
It follows that $\aut(T) \subset \GL_n(Z)$ which is discrete.\\

We then have the sequence:
$$0 \to P \times_G W \to P\times_G G/T \to P\times_G G/N \to 0$$
Since $X$ is simply connected $P \times_G W$ is trivial and we have the desired.
\end{proof}
\begin{defi}
We define the Killing form as:
$B(x,y) = {\rm tr}(ad(x)ad(y))$
For $x,y\in \mathfrak{g}$.

It has a few key properties that we will use. Namely:
\begin{itemize}
    \item That the Killing form of a nilpotent algebra is everywhere $0$.
    \item A Lie algebra is Semi-simple iff the Killing form is non-degenerate over the algebra
    \item $2$ ideals of a Lie algebra have no intersections then they are orthogonal with respect to the Killing form.
\end{itemize}

\end{defi}

Suppose $G$ is a compact reductive Lie group. Writing $\mathfrak{g} = \mathfrak{z} \oplus \mathfrak{s}$ for the abelian and semi-simple parts respectively induces a decomposition of each of the fibers $E_x = E_x^1 \oplus E_x^o$. It suffices to show we can find a regular element in the semi-simple part.\\

Now let $G$ be a compact semi-simple Lie group.
Let $E_k$ be the vector subfibers of $E$ with meromorphic sections of degree at least $k$.
Notice that  $[E_i,E_j] \subset E_{i+j}$ by counting degrees. This implies that elements of $E_1$ are $ad_{\mathfrak{g}}$-nilpotent. Let the sub-algebra defined by $E_1$ be $\mathfrak{g}_1$ and we now have an orthogonal fiber $E_0$, by the Killing form. Let the orthogonal sub-algebra under the Killing form be $\mathfrak{g}_0$.

\begin{lemma}
There is a section in $P \times_G ad$ that is regular at some point.
\end{lemma}
\begin{proof}
Consider the Cartan subalgebras of $\mathfrak{g}_0$. Choose a regular element. Since $\mathfrak{g}_0$ is orthogonal to $\mathfrak{g}_1$, lift it to a global section.
\end{proof}

We need $1$ more lemma before we are finally ready to prove Grothendieck's theorem.
\begin{lemma}
If $G$ is a reductive connected Lie group. There is some finite subgroup, $z$, so that $G/z$ is the product of an abelian and semisimple group.
\end{lemma}
\section{Grothendieck's Theorem}
\begin{theorem}{Classification of Principle Bundles on $\p$}
Let $G$ be a reductive connected Lie group.
The map:
$$H^1(X,O_X(H))/W \to H^1(X,O_X(G))$$
Is a bijection.
\end{theorem}
\begin{proof}
We have seen the surjectivity of it above.
Consider the commutative diagram:

\tikzset{every picture/.style={line width=0.75pt}} %

\begin{center}
\begin{tikzpicture}[x=0.75pt,y=0.75pt,yscale=-1,xscale=1]
\draw    (296,105.22) -- (338,105.22) ;
\draw [shift={(340,105.22)}, rotate = 180] [color={rgb, 255:red, 0; green, 0; blue, 0 }  ][line width=0.75]    (10.93,-3.29) .. controls (6.95,-1.4) and (3.31,-0.3) .. (0,0) .. controls (3.31,0.3) and (6.95,1.4) .. (10.93,3.29)   ;

\draw    (301,168.22) -- (343,168.22) ;
\draw [shift={(345,168.22)}, rotate = 180] [color={rgb, 255:red, 0; green, 0; blue, 0 }  ][line width=0.75]    (10.93,-3.29) .. controls (6.95,-1.4) and (3.31,-0.3) .. (0,0) .. controls (3.31,0.3) and (6.95,1.4) .. (10.93,3.29)   ;

\draw    (240,121.22) -- (240.95,161.22) ;
\draw [shift={(241,163.22)}, rotate = 268.64] [color={rgb, 255:red, 0; green, 0; blue, 0 }  ][line width=0.75]    (10.93,-3.29) .. controls (6.95,-1.4) and (3.31,-0.3) .. (0,0) .. controls (3.31,0.3) and (6.95,1.4) .. (10.93,3.29)   ;

\draw    (395,120.22) -- (395.95,158.22) ;
\draw [shift={(396,160.22)}, rotate = 268.57] [color={rgb, 255:red, 0; green, 0; blue, 0 }  ][line width=0.75]    (10.93,-3.29) .. controls (6.95,-1.4) and (3.31,-0.3) .. (0,0) .. controls (3.31,0.3) and (6.95,1.4) .. (10.93,3.29)   ;

\draw (245,104) node   {$H^{1}( X,O_{x}( H))$};
\draw (391,103) node   {$H^{1}( X,O_{x}( G))$};
\draw (404,170) node   {$H^{1}( X,O_{x}( G/z))$};
\draw (243,172) node   {$H^{1}( X,O_{x}( H/z))$};

\end{tikzpicture}
\end{center}
Suppose $\alpha,\beta \in H^1(X,O_X(H))$ are mapped to the same image in $H^1(X,O_X(G))$.
Looking at the diagram, it is clear that they must have the same image in $H^1(X,O_X(G/z))$ or $H^1(X,O_X(H/z))$. In the first, by $3^{rd}$ isomorphism theorem we have that $\alpha$ and $\beta$ are in the same equivalence class when taking the quotient with the Weyl group: $(N/z)/(H/z) = W$. In the second case we have a contradiction as $H^1(X,z) = 0$ by $z$ being finite and $X$ being connected
inducing a bijection in the first cohomology groups $H^1(X,H)$ and $H^1(X,H/z)$.

\end{proof}
\chapter{Acknowledgments}
We would like to thank Stephen New for his patience when explaining topics in differential geometry. He taught me all the Lie theory I know and suggested this topic as a capstone for the course in Compact Lie Theory.

\newpage
\printbibliography
\nocite{*}
\end{document}